\documentclass[11pt, reqno]{amsart}

\usepackage{amscd,amsmath}
\usepackage{mathrsfs}
\usepackage{amsfonts}
\usepackage{amssymb}
\usepackage{enumerate}
\usepackage{setspace}
\usepackage{color}
\usepackage{esint}

\usepackage{tikz} 

\usepackage[margin=1in]{geometry}

\usepackage[english]{babel}

\usepackage{tabu}

\usepackage[colorlinks=true, citecolor=blue]{hyperref}

\newtheorem{theorem}{Theorem}[section]
\newtheorem{lemma}{Lemma}[section]

\newtheorem{remark}{Remark}[section]

\newcommand{\eqn}{\begin{eqnarray}}
\newcommand{\een}{\end{eqnarray}}

\newcommand{\ZZ}{\mathbb{Z}}

\newcommand{\pax}{\partial_x}
\definecolor{luh-dark-blue}{rgb}{0.0, 0.313, 0.608}

\numberwithin{equation}{section}

\begin{document}

\title[Singularity formation for the SGN equations]{Singularity formation for the Serre-Green-Naghdi equations and applications to \emph{abcd-}Boussinesq systems}

\author{Hantaek Bae}
\address{Department of Mathematical Sciences, Ulsan National Institute of Science and Technology (UNIST), Republic of Korea}
\email{hantaek@unist.ac.kr}

\author{Rafael Granero-Belinch\'on}
\address{Departamento de Matem\'aticas, Estad\'istica y Computaci\'on, Universidad de Cantabria, Spain.}
\email{rafael.granero@unican.es}

\date{\today}



\begin{abstract}
In this work we prove that the solution of the Serre-Green-Naghdi equation cannot be globally defined when the interface reaches the impervious bottom tangentially. As a consequence, our result complements the paper \emph{Camassa, R., Falqui, G., Ortenzi, G., Pedroni, M., \& Thomson, C. Hydrodynamic models and confinement effects by horizontal boundaries. Journal of Nonlinear Science, 29(4), 1445-1498, 2019.} Furthermore, we also prove that the solution to the $abcd-$Boussinesq system can change sign in finite time. Finally, we provide with a proof of a scenario of finite time singularity for the $abcd-$Boussinesq system. These latter mathematical results are related to the numerics in \emph{Bona, \& Chen, Singular solutions of a Boussinesq system for water waves. J. Math. Study, 49(3), 205-220, 2016}.
\end{abstract}

\maketitle

{\small \tableofcontents}


\section{Introduction}
We consider the evolution of a shallow layer of a perfect fluid in a domain of finite depth with a flat bottom topography. We assume that the impervious flat boundary is located at $z=0$ (see figure \ref{fig1}). In this situation, when the physical parameters are in certain regime, the evolution of the free boundary can be approximated by several asymptotic models leading to approximations with different degree of accuracy. More precisely, let us denote $\epsilon$ the amplitude parameter and $\mu$ the shallowness parameter \cite{MR3060183}
$$
\epsilon=\frac{a}{h_0},\quad \mu^2=\frac{h_0^2}{L^2},
$$
where $a$ is the size of the surface variation, $L$ is the horizontal length scale (such as the wavelength of the wave) and $h_0$ is the typical depth. 

We also write $\eta$ for the free surface elevation from the average, $h=1+\epsilon \eta$ is a function representing the interface of the fluid and $u$ is the layer-averaged horizontal velocity of the fluid (see figure \ref{fig1})

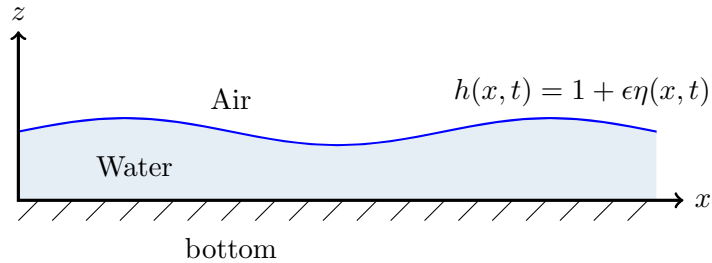
\begin{figure}[h]\label{fig1}
\begin{center} 
\begin{tikzpicture}[domain=0:3*pi, scale=0.9] 
\draw (pi,1.5) node { Air};
\draw[ultra thick, smooth, color=blue] plot (\x,{0.2*sin(\x r)+1});
\fill[luh-dark-blue!10] plot[domain=0:3*pi] (\x,0) -- plot[domain=3*pi:0] (\x,{0.2*sin(\x r)+1});
\draw[very thick,<->] (3*pi+0.4,0) node[right] {$x$} -- (0,0) -- (0,2.5) node[above] {$z$};
\node[right] at (2*pi,1.6) {$h(x,t)=1+\epsilon\eta(x,t)$};
\node[right] at (1,0.5) {Water};
\draw[-] (0,-0.3) -- (0.3, 0);
\draw[-] (0.5,-0.3) -- +(0.3, 0.3);
\draw[-] (1,-0.3) -- +(0.3, 0.3);
\draw[-] (1.5,-0.3) -- +(0.3, 0.3);
\draw[-] (2,-0.3) -- +(0.3, 0.3);
\draw[-] (2.5,-0.3) -- +(0.3, 0.3);
\draw[-] (3,-0.3) -- +(0.3, 0.3);
\draw[-] (3.5,-0.3) -- +(0.3, 0.3);
\draw[-] (4,-0.3) -- +(0.3, 0.3);
\draw[-] (4.5,-0.3) -- +(0.3, 0.3);
\draw[-] (5,-0.3) -- +(0.3, 0.3);
\draw[-] (5.5,-0.3) -- +(0.3, 0.3);
\draw[-] (6,-0.3) -- +(0.3, 0.3);
\draw[-] (6.5,-0.3) -- +(0.3, 0.3);
\draw[-] (7,-0.3) -- +(0.3, 0.3);
\draw[-] (7.5,-0.3) -- +(0.3, 0.3);
\draw[-] (8,-0.3) -- +(0.3, 0.3);
\draw[-] (8.5,-0.3) -- +(0.3, 0.3);
\draw[-] (9,-0.3) -- +(0.3, 0.3);
\draw (pi,-0.7) node { bottom};
\end{tikzpicture}   
\end{center}
\caption{The water-air interface $h$.}
\end{figure}

Then, some asymptotic models that one can consider as valid approximations of the full water wave problem are the nonlinear shallow water equations ($O(\mu)$ approximation), 
\begin{subequations}\label{NSW}
\begin{align}
&\eta_t+\pax (hu)=0, \label{NSWa}\\
&u_t +\epsilon u\pax u+\pax \eta =0, \label{NSWb}
\end{align}
\end{subequations}
the Serre-Green-Naghdi equations (SGN in short) \cite{MR3060183} ($O(\mu^2)$ approximation)
\begin{subequations}\label{GN}
\begin{align}
&\eta_t+\pax (hu)=0, \label{GN1a}\\
&u_t-\frac{\mu}{3h}\partial_x \left(h^3\partial_x u_t\right) +\epsilon u\pax u+\pax \eta -\frac{\epsilon\mu}{3h}\partial_x\left(h^3\left(u\partial_x^2 u-(\partial_x u)^2\right)\right)=0, \label{GN1b}
\end{align}
\end{subequations}
or the $abcd-$Boussinesq system \cite{bona1976model,bona2002boussinesq,bona2004boussinesq,bona2005long} ($O(\mu^2)$ approximation for weakly nonlinear waves $\epsilon=O(\mu)$)
\begin{subequations}\label{abcdB}
\begin{align}
&\left(1-b\mu\pax^{2}\right)\frac{h_{t}}{\epsilon} +\pax (a\mu\pax^{2}u+hu)=0,\label{abcdB1}\\
&\left(1-d\mu\pax^{2}\right)u_{t} +\pax (c\mu\pax^{2}h+\eta)+\epsilon u\pax u=0.\label{abcdB2}
\end{align}
\end{subequations}

The SGN equations are an extension of the classical shallow water (or Saint-Venant) equations taking into account first-order contributions of the dispersion of the water-waves problem. The SGN were derived by Serre \cite{serre1953contribution} and Green \& Naghdi \cite{green1976derivation} and heavily studied ever since. We refer the reader to the works \cite{MR3817287,MR2753374,MR2481064,MR2400253,MR3493692,MR3291451,MR2906225}
 and the references therein. 

One of the recent questions posed for the SGN system is the so-called shoreline problem \cite{MR3817287}. There, the water surface hits the bottom topography at the (a priori unknown) point $X(t)$ in a transverse way such that
$$
\partial_x h(X(t),t)>0.
$$
This is a challenging free boundary problem. In a related paper, Camassa, Falqui, Ortenzi, Pedroni \& Thomson \cite{camassa2019hydrodynamic} studied confinement effects by rigid boundaries in the dynamics of ideal fluids. These authors focused on the consequences of establishing contacts of material surfaces with the confining boundaries. When contact happens, these authors showed that certain self-similar solutions to the nonlinear shallow water equations \eqref{NSW} can develop singularities in finite time. These singularities are related to the well-known fact that the fluid will try to fill the dry spot. We observe that wetting (the function $h=1+\epsilon\eta$ becoming strictly positive) can only happen with loss of derivatives. In other words, water surface cannot detach from the bottom before a singularity occurs. The authors in \cite{camassa2019hydrodynamic} considered a class of special solutions of \eqref{NSW}, the so-called simple waves, and they proved that the vertical location of the shock position depends on the initial steepness of the interface near the bottom. This paper tries to shed some light on this question when the SGN equations are considered. In particular, we prove that solutions of the SGN equations having initially a dry spot cannot be globally smooth (see theorem \ref{theorem1} below).

The $abcd-$Boussinesq system \eqref{abcdB} was proposed by Bona, Chen \& Saut \cite{bona2002boussinesq,bona2004boussinesq}. In order that the system can be applied as a water wave model in certain regime, the parameters $a,b,c,$ $d$ must satisfy
$$
a+b+c+d=\frac{1}{3}-\tau,
$$
where $\tau$ is the surface tension coefficient. In this work we will always consider the case
$$
b,d\geq0,\quad a,c\leq0.
$$ 
In the case $b=d>0$ and $a,c<0$ the $abcd-$Boussinesq system admits the following Hamiltonian formulation
$$
\left(\begin{array}{c}\eta\\u\end{array}\right)_t=(1-b\partial_x^2)^{-1}\partial_x \left(\begin{array}{cc}0 & 1\\1 & 0\end{array}\right)\delta\mathscr{H}(\eta,u),
$$
with
$$
\mathscr{H}(\eta,u)=\frac{1}{2}\int -c(\partial_x u )^2-a(\partial_x \eta)^2+\eta^2+(1+\eta)u^2dx.
$$
Besides these Hamiltonian systems, there are a number of celebrated models that form part of the family of $abcd-$Boussinesq equations. For instance, the Bona-Smith system \cite{bona1976model}, the Kaup system \cite{kaup1975higher} or the Bona-Chen system \cite{bona1998boussinesq}.

The literature on $abcd-$Boussinesq systems is huge and as a consequence our references cannot be exhaustive. We first want to mention the papers \cite{bona1976model,schonbek1981existence,
kwak2019scattering,kwak2019asymptotic} where the authors study the global existence of solutions. Regarding the possible occurrence of singularities, Bona \& Chen \cite{bona2016singular} studied this system numerically. These authors obtained numerical evidence of solutions losing their positivity, becoming negative and finally blowing up in finite time. There are results showing the existence of solution on extended time intervals \cite{burtea2016new,burtea2016long}. Remarkably, \cite{burtea2016new} considers solutions that are not strictly positive and studies the case without the non-cavitation condition. The interested readers can go to \cite{MR3060183,lannesreview} and the references therein to find more details on the state of the art for the $abcd-$Boussinesq system. One of the goals of this work is to provide a rigorous mathematical foundation to the numerical evidence presented in \cite{bona2016singular}. In particular, we will be able to show the existence of solutions losing the positivity in finite time (see Theorem \ref{theorem2} below). Furthermore, we will provide with a scenario for finite time blow up of the type
$$
\limsup_{t\rightarrow T^*}\|h\|_{L^\infty}+\|u\|_{\dot{C}^{1/2}}=\infty
$$
(see Theorem \ref{theorem3} below). We would like to emphasize that, since we do not have any control on the time interval where our hypotheses occur, the comparison with the long time existence results in \cite{burtea2016new} is far from trivial.

The plan of the paper is as follows: In section \ref{sec2} we present our main results together with a brief discussion. In section \ref{sec3} we prove that the solution to the SGN equations \eqref{GN} emanating from certain initial data cannot be globally smooth. Finally, in section \ref{sec4} we prove that the solution to certain $abcd-$Boussinesq systems become negative in finite time. Furthermore, we also present a rigorous scenario for finite time singularity in the case of solutions to $abcd-$Boussinesq systems that are negative in an interval.

\section{Main results}\label{sec2}

We consider the systems \eqref{GN} and \eqref{abcdB} in the domain $\mathbb{T}=[-\pi,\pi]$ with periodic boundary conditions. Since the average depth of the layer is originally $1$, the wave profile has zero mean:
\[
\int_{-\pi}^\pi \eta_0(x)dx=0.
\]
By (\ref{GN1a}) and (\ref{abcdB1}), this condition is propagated over time. Thanks to structure of the equations, we also assume that 
\[
\int_{-\pi}^\pi u(t,x) dx =0.
\]

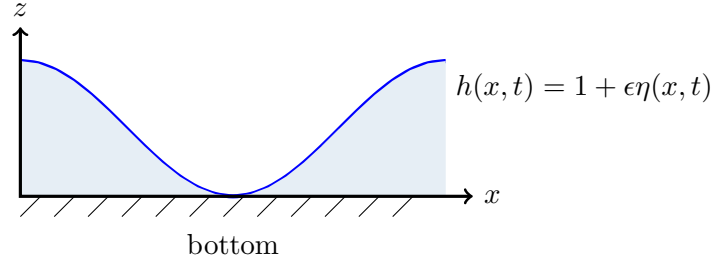
\begin{figure}[h]\label{fig2}
\begin{center} 
\begin{tikzpicture}[domain=0:2*pi, scale=0.9] 
\draw[ultra thick, smooth, color=blue] plot (\x,{1+cos(\x r)});
\fill[luh-dark-blue!10] plot[domain=0:2*pi] (\x,0) -- plot[domain=2*pi:0] (\x,{1+cos(\x r)});
\draw[very thick,<->] (2*pi+0.4,0) node[right] {$x$} -- (0,0) -- (0,2.5) node[above] {$z$};
\node[right] at (2*pi,1.6) {$h(x,t)=1+\epsilon\eta(x,t)$};
\draw[-] (0,-0.3) -- (0.3, 0);
\draw[-] (0.5,-0.3) -- +(0.3, 0.3);
\draw[-] (1,-0.3) -- +(0.3, 0.3);
\draw[-] (1.5,-0.3) -- +(0.3, 0.3);
\draw[-] (2,-0.3) -- +(0.3, 0.3);
\draw[-] (2.5,-0.3) -- +(0.3, 0.3);
\draw[-] (3,-0.3) -- +(0.3, 0.3);
\draw[-] (3.5,-0.3) -- +(0.3, 0.3);
\draw[-] (4,-0.3) -- +(0.3, 0.3);
\draw[-] (4.5,-0.3) -- +(0.3, 0.3);
\draw[-] (5,-0.3) -- +(0.3, 0.3);
\draw[-] (5.5,-0.3) -- +(0.3, 0.3);
\draw (pi,-0.7) node { bottom};
\end{tikzpicture}   
\end{center}
\caption{A scheme of the initial data leading to blow up in Theorem \ref{theorem1}.}
\end{figure}

Without loss of generality, in what follows we consider $\mu=\epsilon=1$. Now we can state our main result:

\begin{theorem}\label{theorem1}
Let us consider a pair of functions $(\eta_0,u_0)$ satisfying that
\begin{enumerate}[]
\item (A1): $1+\epsilon\eta_0(x)\ge 0$.
\item (A2): $\eta_0$ is a smooth even function such that $1+\epsilon\eta_0(0)=0$ and $\pax^2 \eta_0(0)\ge0$.
\item (A3): $u_0$ is a smooth odd function such that $\partial_x u_0(0)<0$.
\end{enumerate}
Let us assume that a solution to \eqref{GN} exists for this initial data. Then, this solution blows up in finite time.
\end{theorem}

\begin{remark}
The same result holds for the nonlinear shallow water equations \eqref{NSW}. In this case, the strict positivity of the function $h=1+\eta$ implies that the nonlinear shallow water equations form a Friedrichs symmetrizable hyperbolic system. Thus, our singularity result in this case is achieved in the regime where the hyperbolicity degenerates.
\end{remark}

To prove this theorem we consider initial data with certain symmetries and prove the finite time blow-up using a pointwise method (see Theorem \ref{theorem1}). In other words, we prove that for certain family of initial data, if the solution exists, then certain derivatives for $h=1+\eta$ and $u$ blow up at the origin $x=0$.

We observe that the problem that we consider here is somehow related to the shoreline problem that studies the motion of water at a beach  \cite{MR3817287} but where the surface of the water is not transverse to the bottom topography at the shoreline. 

A similar idea can be applied to the $abcd$-Boussinesq system \eqref{abcdB}. For this system we can prove the following result:

\begin{theorem}\label{theorem2}
Let $b=c=0$ and consider $a\leq 0$ and $d\geq0$. Then there are smooth initial data $(\overline{h}_{0},\overline{u}_{0})$ with $\overline{h}_{0}>0$ such that the corresponding solution $h$ to (\ref{abcdB}) becomes negative in finite time.
\end{theorem}

Once the function $h$ becomes negative it loses its physical meaning. Moreover, once $h$ becomes negative, under certain circumstances, the solution to this system may end up blowing up. The following theorem establishes some sufficient conditions leading to a finite time blow-up:
\begin{theorem}\label{theorem3}
Let $b=d=0$ and consider $a,c\leq0$. Let $(h,u)$ be a smooth solution having a maximal lifespan $T_{\text{max}}\in(0,\infty]$. Assume that $h$ is an even function and $u$ is an odd function. Furthermore assume that there exists $0<\omega\in (0,\pi]$ and $\sigma>0$ such that this solution satisfies
\begin{enumerate}
\item 
$$
h<-\sigma \quad \forall (x,t)\in [0,\omega]\times[0,T_{\text{max}}),
$$
\item 
$$
u\leq0 \quad \forall (x,t)\in [0,\omega]\times[0,T_{\text{max}}),
$$
\item 
$$
\partial_x h\geq0 \quad \forall (x,t)\in [0,\omega]\times[0,T_{\text{max}}),
$$
\item
$$
c\partial_x^3h\geq0 \quad \forall (x,t)\in [0,\omega]\times[0,T_{\text{max}}),
$$
\item
$$
a\partial_x^2u\geq0 \quad \forall (x,t)\in [0,\omega]\times[0,T_{\text{max}}).
$$
Then 
$$
T_{\text{max}}<\infty.
$$
\end{enumerate}
\end{theorem}
\begin{remark}
We observe that these values for the parameters $a,b,c,d$ imply that surface tension effects are considered.
\end{remark}
\begin{remark}
These hypotheses mean that $h$ is negative and has a single minimum in the interval $[-\omega,\omega]$. They also imply that $u$ is decreasing on $[0,\omega]$. They seem to be verified at least qualitatively by the numerical solution in \cite[Figure 11]{bona2016singular}.
\end{remark}
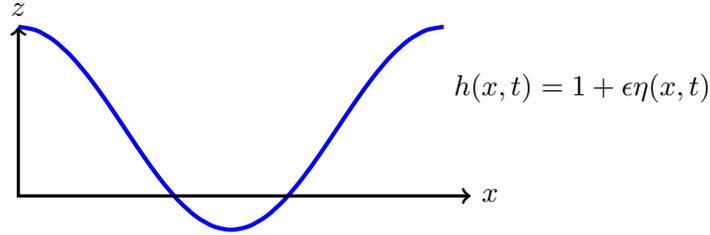
\begin{figure}[h]\label{fig3}
\begin{center} 
\begin{tikzpicture}[domain=0:2*pi, scale=0.9] 
\draw[ultra thick, smooth, color=blue] plot (\x,{1+1.5*cos(\x r)});
\draw[very thick,<->] (2*pi+0.4,0) node[right] {$x$} -- (0,0) -- (0,2.5) node[above] {$z$};
\node[right] at (2*pi,1.6) {$h(x,t)=1+\epsilon\eta(x,t)$};
\end{tikzpicture}   
\end{center}
\caption{A scheme of the initial data leading to blow up in Theorem \ref{theorem3}.}
\end{figure}

This result gives conditions leading to a loss of derivatives for the solution and as a consequence provides us with a better picture of the possible singularity. In other words, this theorem states that a finite time singularity will occur for solutions having certain geometry. We want to remark that \emph{a priori} not every solution that becomes negative will end up in a finite time singularity. As in the case of the Muskat problem \cite{cordoba2015note,cordoba2017note}, the solution could potentially become positive again before the singularity actually occurs.  Finally, let us also mention that this is a somehow unsatisfactory result in the sense that it is based on assumptions about the solutions which are not known to hold from conditions on the initial data.

\section{The Serre-Green-Naghdi equations}\label{sec3}
\begin{proof}[{\bf Proof of Theorem \ref{theorem1}}]
We now chose initial data leading to finite time blow-up. We consider initial data such that
\begin{enumerate}[]
\item (A1): $1+\eta_0(x)\ge 0$.
\item (A2): $\eta_0$ is a smooth even function such that $1+\eta_0(0)=0$ and $\pax^2 \eta_0(0)\ge0$.
\item (A3): $u_0$ is a smooth odd function such that $\partial_x u_0(0)<0$.
\end{enumerate}
Now let us assume that there exists a unique smooth solution of \eqref{GN} starting from such initial data. In the $(h,u)$ formulation, \eqref{GN} reads
\begin{subequations}\label{GN2}
\begin{align}
&h_t+\pax (hu)=0, \label{GN2a}\\
&u_t-\frac{1}{3h}\partial_x \left(h^3\partial_x u_t\right) +  u\pax u+\pax h -\frac{1}{3h}\partial_x\left(h^3\left(u\partial_x^2 u-(\partial_x u)^2\right)\right)=0.\label{GN2b}
\end{align}
\end{subequations}
We note that the symmetries of $(h,u)$ are preserved as long as the smooth solution of (\ref{GN2}) exists. 

We now define the following quantities:
\[
\alpha_i(t)=\pax^i h(x,t)\bigg{|}_{x=0},\quad \beta_i(t)=\pax^i u(x,t)\bigg{|}_{x=0}.
\]
Due to the symmetry properties of $h$ and $u$, we first have 
\[
\beta_0(t)=\beta_2(t)=\alpha_1(t)=\alpha_3(t)=0.
\]
We next derive the equations of $\alpha_{0}$, $\alpha_{2}$, and $\beta_{1}$. First we find that 
\[
\frac{d}{dt}\alpha_0(t)=-\alpha_1(t)\beta_0(t)-\alpha_0(t)\beta_1(t)=-\alpha_0(t)\beta_1(t).
\]
Since $\alpha_0(0)=0$, we have that $\alpha_{0}(t)=0$. Similarly, 
\eqn \label{alpha2}
\frac{d}{dt}\alpha_2(t)=-3\beta_1(t)\alpha_2(t).
\een 
To find the ODE for $\beta_1$, we take $\pax$ to (\ref{GN2b}) and evaluate the resulting equation at $x=0$. Then, we obtain
\begin{equation}\label{beta1}
\begin{split}
\frac{d}{dt}\beta_1(t)&=-\beta^{2}_1(t)-\alpha_2(t).
\end{split}
\end{equation}
Since $\alpha_{2}(0)\ge 0$ we obtain that $\alpha_{2}(t)\ge 0$. This in turn implies 
\[
\frac{d}{dt}\beta_1(t)\leq -\beta^{2}_1(t).
\]
Solving this ODE, we derive the following inequality:
\[
\beta_{1}(t)\leq \frac{\beta_{1}(0)}{1+t\beta_{1}(0)}.
\]
Since $\beta_{1}(0)<0$, 
\[
\beta_{1}(t)\rightarrow -\infty \quad \text{as} \quad t\rightarrow -\frac{1}{\beta_{1}(0)}.
\]
This completes the proof.
\end{proof}

\begin{remark}
It may seem that the blow-up is caused by $\beta_1^2$, but this is actually not the case. To show this, we consider the following system where the term $\beta_1^2$ has been dropped: \begin{equation}\label{alpha2b}
\begin{split}
\frac{d}{dt}\alpha_2(t)&=-3\beta_1(t)\alpha_2(t),\\
\end{split}
\end{equation}
\begin{equation}\label{beta1b}
\begin{split}
\frac{d}{dt}\beta_1(t)&=-\alpha_2(t).
\end{split}
\end{equation}
This system still has a finite time blow-up. Indeed, we have 
$$
\frac{d}{dt}\alpha_2(t)=-3\beta_1(t)\alpha_2(t)=3\beta_1(t)\frac{d}{dt}\beta_1(t),
$$
so
$$
\alpha_2(t)-\alpha_2(0)=\frac{3}{2}\left(\beta_1(t)^2-\beta_1(0)^2\right).
$$
Then, \eqref{beta1b} implies 
$$
\frac{d}{dt}\beta_1(t)=-\alpha_2(0)-\frac{3}{2}\left(\beta_1(t)^2-\beta_1(0)^2\right)=-\frac{3}{2}\beta_1(t)^2 +\left(\frac{3}{2}\beta_1(0)^2-\alpha_2(0)\right).
$$
If 
\[
\frac{3}{2}\beta_1(0)^2-\alpha_2(0)\leq 0,
\]
the same conclusion is reached.
\end{remark}

\begin{remark}
A slight modification of the proof allows to prove the finite time blow-up for initial data  satisfying $\alpha_0(0)=\alpha_1(0)=\beta_0(0)=0$, without imposing the symmetry conditions (A2) and (A3).  
\end{remark}

\section{Application to the $abcd$-Boussinesq system}\label{sec4}
We start this section defining the functional framework: for $\tau>0$, we define the following Banach scale of spaces of analytic functions
$$
X_\tau=\left\{f:\sum_{k\in\ZZ} e^{|k|\tau}|\widehat{f}(k)|<\infty\right\}.
$$
Functions in this space are analytic in a complex strip of width $\tau$ around the real axis.

In this section, we apply the method used for (\ref{GN}) to (\ref{abcdB}). To do this, we frist prove the local existence of analytic solutions. To prove the local existence of unique solutions of (\ref{abcdB}), we use a method akin to the abstract Cauchy-Kovalevski theorem. We collect this result in the following Lemma: 

\begin{lemma}\label{lemma2}Let $(h_0,u_0)\in X_1\times X_1$ and $a,b,c$ and $d$ be constants such that either
\begin{itemize}
\item $$b,d>0, a,c\in\mathbb{R},$$ or
\item $$b=c=0, a\in\mathbb{R},d>0.$$
\end{itemize}
There exists $0<T=T(a,b,c,d,\|h_0\|_{X_{1}},\|u_0\|_{X_{1}})$ such that there exists a unique solution of (\ref{abcdB}) in the following space
$$
(h,u)\in C([-T,T],X_{0.25})\times C([-T,T],X_{0.25}).
$$
\end{lemma}
\begin{proof} Let us assume that $t$ moves forward. The case of backward time can be handled similarly. Without loss of generality we fix $\mu=\epsilon=1$. We are going to prove the a priori estimates. We define
$$
\nu(t)=0.9-k|t|,
$$
for certain $k$ to be fixed later. The Fourier transform of (\ref{abcdB}) is
\begin{subequations}\label{abcdBv2}
\begin{align}
&\hat{h}_{t}(n)=\frac{ia n^3}{1+b|n|^{2}} \hat{u}(n)-\frac{in}{1+b|n|^{2}} \sum_{m=-\infty}^\infty \hat{h}(n-m)\hat{u}(m),\label{abcdB1v2}\\
&\hat{u}_{t}(n) =\frac{ci n^3}{1+d|n|^{2}} \hat{h}(n)-\frac{in}{1+d|n|^{2}}\hat{h}(n)- \frac{in}{2(1+d|n|^{2})} \sum_{m=-\infty}^\infty \hat{u}(n-m)\hat{u}(m)=0.\label{abcdB2v2}
\end{align}
\end{subequations}
We now consider the first case $b,d>0, a,c\in\mathbb{R}$. We define the energy
$$
E(t)=\|h(t)\|_{X^{\nu(t)}}+\|u(t)\|_{X^{\nu(t)}}.
$$
Then, using that these spaces satisfy the following inequality
$$
\|fg\|_{X^{\nu(t)}}\leq \|f\|_{X^{\nu(t)}}\|g\|_{X^{\nu(t)}},
$$ 
we compute
\begin{align*}
\frac{d}{dt}E(t)&=\sum_{n=-\infty}^\infty e^{\nu(t)|n|}\Re( \hat{h}_t(n)\frac{\overline{\hat{h}}(n)}{|\hat{h}(n)|}+\hat{u}_t(n)\frac{\overline{\hat{u}}(n)}{|\hat{u}(n)|})-k\sum_{n=-\infty}^\infty|n|e^{\nu(t)|n|}(|\hat{h}(n)|+|\hat{u}(n)|)\\
&\leq \frac{|a|}{b}\sum_{n=-\infty}^\infty|n|e^{\nu(t)|n|}|\hat{u}(n)| -k\sum_{n=-\infty}^\infty|n|e^{\nu(t)|n|}(|\hat{h}(n)|+|\hat{u}(n)|)+C(b)\|h(t)\|_{X^{\nu(t)}}\|u(t)\|_{X^{\nu(t)}}\\
&\quad+\frac{|c|}{d}\sum_{n=-\infty}^\infty|n|e^{\nu(t)|n|}|\hat{h}(n)| +C(d)\|h(t)\|_{X^{\nu(t)}}+C(d)\|u(t)\|_{X^{\nu(t)}}^2.
\end{align*}
We now take
$$
k=\max\left\{\frac{|a|}{b},\frac{|c|}{d}\right\}+1<\infty.
$$
With this choice of $k$ we find that
\begin{align*}
\frac{d}{dt}E(t)&\leq \left(\frac{|a|}{b}-k\right)\sum_{n=-\infty}^\infty|n|e^{\nu(t)|n|}|\hat{u}(n)| +C(b)\|h(t)\|_{X^{\nu(t)}}\|u(t)\|_{X^{\nu(t)}}\\
&\quad+\left(\frac{|c|}{d}-k\right)\sum_{n=-\infty}^\infty|n|e^{\nu(t)|n|}|\hat{h}(n)| +C(d)\|h(t)\|_{X^{\nu(t)}}+C(d)\|u(t)\|_{X^{\nu(t)}}^2\\
&\leq C(d)E(t)+C(b,d)E^2(t).
\end{align*}
Using Gronwall's inequality we conclude the existence of a uniform time of existence $T=T(a,b,c,d,E(0))$ such that 
$$
E(t)<2E(0).
$$
Restricting to even smaller times we can ensure that $\nu(t)>0.25$.

We now consider the second case $b=c=0, a\in\mathbb{R}, d>0$. With this choice of parameters, we have that the abcd-Boussinesq equation reads
\begin{subequations}\label{abcdBv3}
\begin{align}
&\hat{h}_{t}(n)=ia n^3 \hat{u}(n)-in\sum_{m=-\infty}^\infty \hat{h}(n-m)\hat{u}(m),\label{abcdB1v3}\\
&\hat{u}_{t}(n) =-\frac{in}{1+d|n|^{2}}\hat{h}(n)- \frac{in}{2(1+d|n|^{2})} \sum_{m=-\infty}^\infty \hat{u}(n-m)\hat{u}(m).\label{abcdB2v3}
\end{align}
\end{subequations}
We define the energy
$$
\mathfrak{E}(t)=\|h(t)\|_{X^{\nu(t)}}+\|\pax^2 u(t)\|_{X^{\nu(t)}}.
$$
\begin{align*}
\frac{d}{dt}\mathfrak{E}(t)&=\sum_{n=-\infty}^\infty e^{\nu(t)|n|}\Re( \hat{h}_t(n)\frac{\overline{\hat{h}}(n)}{|\hat{h}(n)|}+\widehat{\pax^2 u}_t(n)\frac{\overline{\widehat{\pax^2 u}}(n)}{|\widehat{\pax^2 u}(n)|})-k\sum_{n=-\infty}^\infty|n|e^{\nu(t)|n|}(|\hat{h}(n)|+|n|^2|\hat{u}(n)|)\\
&\leq |a|\sum_{n=-\infty}^\infty e^{\nu(t)|n|}|n|^3|\hat{u}(n)|+\|\pax h(t)\|_{X^{\nu(t)}}\|u(t)\|_{X^{\nu(t)}}+\|h(t)\|_{X^{\nu(t)}}\|\pax u(t)\|_{X^{\nu(t)}}\\
&\quad -k\sum_{n=-\infty}^\infty|n|e^{\nu(t)|n|}(|\hat{h}(n)|+|n|^2|\hat{u}(n)|)+\frac{1}{d}\sum_{n=-\infty}^\infty|n|e^{\nu(t)|n|}|\hat{h}(n)|+\frac{1}{d}\|u(t)\|_{X^{\nu(t)}}\|\pax u(t)\|_{X^{\nu(t)}}
\end{align*}
We now take
$$
k=\max\left\{|a|,\frac{2}{d}+4\mathfrak{E}(0)\right\}+1<\infty.
$$
Then
\begin{align*}
\frac{d}{dt}\mathfrak{E}(t)&\leq (|a|-k)\sum_{n=-\infty}^\infty e^{\nu(t)|n|}|n|^3|\hat{u}(n)|+\|h(t)\|_{X^{\nu(t)}}\|\pax u(t)\|_{X^{\nu(t)}}\\
&\quad +\left(\frac{1}{d}-\frac{k}{2}\right)\sum_{n=-\infty}^\infty|n|e^{\nu(t)|n|}|\hat{h}(n)|+\frac{1}{d}\|u(t)\|_{X^{\nu(t)}}\|\pax u(t)\|_{X^{\nu(t)}}\\
&\quad+\left(\|u(t)\|_{X^{\nu(t)}}-\frac{k}{2}\right)\sum_{n=-\infty}^\infty|n|e^{\nu(t)|n|}|\hat{h}(n)|\\
&\leq C(d)\mathfrak{E}(t)^2.
\end{align*}
Invoking Gronwall's inequality, we conclude the existence of a uniform time of existence $T=T(a,b,c,d,\mathfrak{E}(0))$ such that 
$$
\mathfrak{E}(t)<2\mathfrak{E}(0).
$$
Restricting to even smaller times we can ensure that $\nu(t)>0.25$.

The uniqueness follow from a standard contradiction argument once that the constructed solution is analytic.
\end{proof}

Equipped with this local existence of solution result, we can prove the existence of solution that change sign.

\begin{proof}[{\bf Proof of Theorem \ref{theorem2}}] 
We now chose initial data leading to solutions that change sign at the origin $x=0$. Let $h_0\in X_1$ be an even function and $u_0\in X_1$ be an odd function. We also impose the following hypothesis: 
\begin{enumerate}[]
\item (A1): $a\partial_x^3 u_0(0)>0$.
\item (A2): $h_0(0)=0$.
\item (A3): $h_0(x)>0$ if $x\neq0$ and $h_0(x)>1$ for $0.5<|x|<\pi$.
\item (A4): $\pax^2h_0(x)>\sigma>0$ for $|x|<0.5$.
\end{enumerate}
Due to Lemma \ref{lemma2}, there is a smooth solution emanating from this initial data. We also observe that the symmetries are preserved by the evolution of (\ref{abcdB}). We recall that (\ref{abcdB}) with $b=0$, $\epsilon=1$ and $h=1+\eta$:
\begin{subequations}\label{acdB}
\begin{align}
&h_{t} +\pax (a\pax^{2}u+hu)=0,\\
&\left(1-d\pax^{2}\right)u_{t} +\pax h +u\pax u=0.
\end{align}
\end{subequations}

We denote
\[
\alpha_i(t)=\pax^i h(x,t)\bigg{|}_{x=0},\quad \beta_i(t)=\pax^i u(x,t)\bigg{|}_{x=0}.
\] 
Due to the symmetry properties of $h$ and $u$, we first have 
\[
\beta_0(t)=\beta_2(t)=\alpha_1(t)=\alpha_3(t)=0.
\]
We next consider $\alpha_{0}$:
$$
\frac{d}{dt}\alpha_0(t)=-a\beta_3(t)-\alpha_1(t)\beta_0(t)-\alpha_0(t)\beta_1(t)=-a\beta_3(t)-\alpha_0(t)\beta_1(t).
$$
By (A1) and (A2),  
$$
\frac{d}{dt}\alpha_0(t)\bigg{|}_{t=0^{+}}=-a\beta_3(0)<0.
$$
Similarly, going backward in time (which we can do due to Lemma \ref{lemma2}), 
$$
\frac{d}{dt}\alpha_0(t)\bigg{|}_{t=0^{-}}=a\beta_3(0)>0.
$$
We now chose $\delta\in (0,T)$ which is small enough such that 
\eqn \label{restrictiondelta}
\alpha_{0}(\delta)=h(\delta,0)<0 \quad \text{and} \quad \alpha_{0}(-\delta)=h(-\delta,0)>0.
\een
And we define the new initial data
\eqn \label{newdata}
(\overline{h}_0(x),\overline{u}_0(x))=(h(-\delta,x),u(-\delta,x)).
\een
By restricting the size of $\delta$ if necessary, this new initial data satisfies
$$
\overline{h}_0(x)>0.5 \ \text{ for } \ 0.5<|x|<\pi,\quad \;0<\frac{\sigma}{2}<\partial_x^2\overline{h}_0(x) \ \text{ for } \  |x|<0.5
$$
due to the smoothness of the solution with (A3) and (A4). So, there are no other relative maxima or minima in the interval $(-0.5,0.5)$ and thus $x=0$ is the global minimum for the initial data. Let 
\[
\left(\overline{h}(t,x),\overline{u}(t,x)\right)=\left(h(t-\delta,x),u(t-\delta,x)\right).
\]
Then, $\left(\overline{h}(t,x),\overline{u}(t,x)\right)$ is the solution of (\ref{acdB}) with initial data in (\ref{newdata}). By (\ref{restrictiondelta}), we conclude that $\overline{h}_0(0)>0$ but $\overline{h}(2\delta,0)<0$. 
\end{proof}

\begin{remark}For the case $b\neq0$, we can prove a similar result for initial data satisfying the hypothesis
$$
\frac{a}{b}\pax u_0(0)+\sum_{n=-\infty}^\infty\int_{-\pi}^{\pi}\frac{e^{-\frac{|y+2\pi n|}{\sqrt{b}}}}{b^{1/2}}\left[\pax(h_0(y) u_0(y))+a\pax u_0(y)\right]dy<0.
$$
\end{remark}
Furthermore, once that the solution change sign it may end up having a singularity. In the following theorem we prove that, under certain circumstances, such a singularity occurs.
\begin{proof}[{\bf Proof of Theorem \ref{theorem3}}]
For $0<\lambda<1/2$, we define the functionals
$$
\mathscr{F}_h(t)=-\int_{0}^\omega\frac{h(x,t)}{x^{\lambda}}dx,
$$
$$
\mathscr{F}_u(t)=-\int_{0}^\omega\frac{u(x,t)}{x^{\lambda}}dx.
$$
We observe that, due to the hypotheses of the Theorem, we have that
$$
0<\mathscr{F}_u,\mathscr{F}_h\quad\forall\,t\in [0,T_{\text{max}}).
$$
Furthermore,
$$
\mathscr{F}_h\leq C(\omega,\lambda)\|h\|_{L^\infty}
$$
and
$$
\mathscr{F}_u\leq C(\omega)\|u\|_{\dot{C}^\lambda}.
$$
Thus, a blow up of $\mathscr{F}_u$ or $\mathscr{F}_h$ implies a finite time singularity for the system. We argue by contradiction: let us assume that a solution satisfying the hypotheses in the statement exists globally, \emph{i.e.} $T_{\text{max}}=\infty$. Then we will prove that there is a blow up for our functionals and this will be the contradiction. 

Using the symmetry of $u$ to cancel some of the boundary terms together with the hypotheses of the statement, we compute 
\begin{align*}
\frac{d}{dt}\mathscr{F}_h&=\int_0^\omega \frac{a\pax^{3}u}{x^\lambda}dx+\int_{0}^\omega\frac{\pax(hu)}{x^\lambda}dx\\
&= \frac{a\pax^{2}u}{x^\lambda}\bigg{|}^\omega_0-\int_0^\omega a\pax^{2}u\pax\left(\frac{1}{x^\lambda}\right)dx+\frac{hu}{x^\lambda}\bigg{|}^\omega_0-\int_{0}^\omega hu\pax\left(\frac{1}{x^\lambda}\right)dx\\
&= \frac{a\pax^{2}u(\omega)}{\omega^\lambda}+\lambda\int_0^\omega \frac{a\pax^{2}u}{x^{1+\lambda}}dx+\frac{h(\omega)u(\omega)}{\omega^\lambda}+\lambda\int_{0}^\omega \frac{hu}{x^{1+\lambda}}dx\\
&\geq\lambda\int_{0}^\omega \frac{h}{x^{\lambda}}\frac{u}{x^{\lambda}}\frac{1}{x^{1-\lambda}}dx\\
&\geq\frac{\lambda\sigma}{\omega}\mathscr{F}_u.
\end{align*}
Similarly, using Jensen's inequality we find that
\begin{align*}
\frac{d}{dt}\mathscr{F}_u&=\int_0^\omega \frac{c\pax^{3}h}{x^\lambda}dx+\int_0^\omega \frac{\pax h}{x^\lambda}dx+\int_{0}^\omega\frac{\pax(u^2/2)}{x^\lambda}dx\\
&= \int_0^\omega \frac{c\pax^{3}h}{x^\lambda}dx+\int_0^\omega \frac{\pax h}{x^\lambda}dx+\frac{u^2}{2x^\lambda}\bigg{|}^\omega_0-\int_{0}^\omega \frac{u^2}{2}\pax\left(\frac{1}{x^\lambda}\right)dx\\
&=\int_0^\omega \frac{c\pax^{3}h}{x^\lambda}dx+\int_0^\omega \frac{\pax h}{x^\lambda}dx+\frac{u(\omega)^2}{2\omega^\lambda}+\lambda\int_{0}^\omega \frac{u^2}{2}\frac{1}{x^{1+\lambda}}dx\\
&\geq\lambda\int_{0}^\omega \frac{u^2}{2}\frac{1}{x^{1+\lambda}}dx\\
&\geq\frac{\lambda}{2\omega^{1-\lambda}}\int_{0}^\omega \frac{u^2}{x^{2\lambda}}dx\\
&\geq\frac{\lambda}{2\omega^{2-\lambda}}\left(\mathscr{F}_u\right)^2.
\end{align*}
Thus, we find the blow up in finite time for $\mathscr{F}_u$ and $\mathscr{F}_h$. As a consequence we conclude the contradiction.
\end{proof}

\section*{Acknowledgments}
H.B. was supported by NRF-2018R1D1A1B07049015. R.G-B was supported by the project ”Mathematical Analysis of Fluids and Applications” with reference PID2019-109348GA-I00/AEI/ 10.13039/501100011033 and acronym ``MAFyA” funded by Agencia Estatal de Investigaci\'on and the Ministerio de Ciencia, Innovacion y Universidades (MICIU). The authors are grateful to Vincent Duchene for fruitful discussions and to David Lannes for useful comments that greatly improved the presentation. Finally, the authors gratefully thank to the Referee for the constructive comments and recommendations which definitely help to improve the readability and quality of the paper.

\bibliographystyle{plain}

\end{document}